\documentclass[preprint,12pt]{elsarticle}

\pdfoutput=1

\usepackage{amssymb,amsmath,amsfonts}
\usepackage{amsthm}

\usepackage{graphicx}
\usepackage{url}
\usepackage{subfig}
\usepackage{algorithm,algorithmic}

\usepackage{accents}

\newcommand{\doublehat}[1]{%
    \hat{\underline{#1}}}

\newcommand{\mA}{\mathbf{A}}

\newcommand{\mI}{\mathbf{I}}
\newcommand{\mJ}{\mathbf{J}}
\newcommand{\mQ}{\mathbf{Q}}

\newcommand{\mT}{\mathbf{T}}
\newcommand{\mU}{\mathbf{U}}
\newcommand{\mV}{\mathbf{V}}
\newcommand{\mW}{\mathbf{W}}

\newcommand{\ve}{\mathbf{e}}
\newcommand{\vf}{\mathbf{f}}
\newcommand{\vg}{\mathbf{g}}
\newcommand{\vh}{\mathbf{h}}
\newcommand{\vt}{\mathbf{t}}
\newcommand{\vu}{\mathbf{u}}
\newcommand{\vv}{\mathbf{v}}

\newcommand{\vx}{\mathbf{x}}

\newcommand{\sE}{\mathcal{E}}
\newcommand{\sF}{\mathcal{F}}
\newcommand{\sG}{\mathcal{G}}
\newcommand{\sT}{\mathcal{T}}
\newcommand{\sX}{\mathcal{X}}

\newcommand{\bpi}{\mbox{\boldmath$\pi$}}
\newcommand{\bphi}{\mbox{\boldmath$\phi$}}
\newcommand{\blambda}{\mbox{\boldmath$\lambda$}}

\newcommand{\bmat}[1]{\begin{bmatrix}#1\end{bmatrix}}

\newtheorem{theorem}{Theorem}

\journal{AMC}

\begin{document}

\begin{frontmatter}



\title{A Lanczos Method for Approximating Composite Functions}


\author[label2]{Paul G. Constantine\corref{cor1}}
\ead{paul.constantine@stanford.edu}
\cortext[cor1]{Building 500, Room 501-L, Tel: (650) 723-2330}
\author[label1]{Eric T. Phipps}
\ead{etphipp@sandia.gov}
\fntext[label2]{Sandia National Laboratories is a multi-program laboratory managed
and operated by Sandia Corporation, a wholly owned subsidiary of Lockheed Martin Corporation, for the U.S. Department
of Energy's National Nuclear Security Administration under contract DE-AC04-94AL85000.}

\address[label1]{Sandia National Laboratories\fnref{label2}, Albuquerque, NM 87185}
\address[label2]{Stanford University, Stanford, CA 94305}

\begin{abstract}
We seek to approximate a composite function $h(\vx)=g(f(\vx))$ with a global polynomial. The standard approach chooses
points $\vx$ in the domain of $f$ and computes $h(\vx)$ at each point, which requires an evaluation of $f$ and an
evaluation of $g$. We present a Lanczos-based procedure that implicitly approximates $g$ with a polynomial of $f$. By
constructing a quadrature rule for the density function of $f$, we can approximate $h(\vx)$ using many fewer evaluations
of $g$. The savings is particularly dramatic when $g$ is much more expensive than $f$ or the dimension of $x$
is large. We demonstrate this procedure with two numerical examples: (i) an exponential function composed with a
rational function and (ii) a Navier-Stokes model of fluid flow with a scalar input parameter that depends on multiple
physical quantities.
\end{abstract}


\end{frontmatter}


\section{Introduction \& Motivation}

Many complex multiphysics models employ composite functions, where each member function represents a different physical
model. For example, consider a simple chemical reaction model; the decay of the concentration
depends on the decay rate parameter, but the model for the decay rate (i.e., the Arrhenius model) depends on
the temperature, the gas constant, the activation energy, and the prefactor. The concentration is then a function of the
inputs to the rate model. Studying the effects of these inputs on the concentration could be challenging due to
the large number of combinations of the four rate model inputs. However, one may make use of the composite
structure -- namely, that the concentration depends on these four inputs only through the rate model -- to find a set of
values for the rate parameter with which to study the effects on the concentration. By taking advantage of this
structure, one can hope to reduce the cost of such a study by using fewer evaluations of the models. 

We consider the general setting of a composite function 
\begin{equation}
\label{eq:hdef}
h(\vx)=g(f(\vx)), \qquad \vx=(x_1,\dots,x_d)\in\sX\subset\mathbb{R}^d
\end{equation} 
where
\begin{align}
f&:\sX \;\longrightarrow\; \sF\; \subset\; \mathbb{R}\\
g&:\sF \;\longrightarrow\; \sG\; \subset\; \mathbb{R}.
\end{align}
One may be interested in understanding how $h$ behaves as $\vx$ changes. However, if evaluating $h$ is computationally
expensive, then studies that require many evaluations may be infeasible. A common approach to this situation
is to construct a surrogate approximation of $h$ whose evaluations are cheaper given inputs $\vx$. With
a fixed number of $h$ evaluations, one determines the parameters (e.g., coefficients) of the surrogate model, and then
the surrogate is used to study the behavior of $h$ with respect to changes in $\vx$. 

Global polynomials~\cite{Xiu02,Xiu05} in $\vx$ are popular models for surrogate functions due to their rapid
convergence rate as more terms are added to the approximation. For smooth functions, this often translates to
relatively few evaluations to construct an accurate approximation. The approximation properties of univariate
polynomials ($d=1$) are well-studied. Multivariate ($d>1$) polynomial approximations are less well understood, but a
simple tensor product construction extends the univariate approximation characteristics to the multivariate setting.
Unfortunately, such tensor product approximations suffer from the so-called curse of dimensionality. Loosely speaking,
the number of function evaluations required to construct an accurate approximation increases exponentially as the
dimension of $\vx$ increases. When evaluations of $h$ are expensive, this often precludes the use of a multivariate
polynomial surrogate.

In this work, we propose a strategy that takes advantage of the composite structure of $h$ to reduce the overall
cost of building a multivariate polynomial surrogate. 
Suppose one needs $m$ points in $\sX$ to construct a multivariate polynomial surrogate of $h(\vx)$, where $m$ may be an
exponential function of $d$. Then each evaluation of $h(\vx)$ requires first an evaluation of $f(\vx)$ followed by an
evaluation of $g(f)$, i.e., $2m$ total function evaluations. 

Our proposed strategy computes the same $m$ evaluations of $f$, but uses them to
implicitly approximate a univariate density function on the range space $\sF$. We then find $k \ll m$ points in $\sF$
on which to evaluate $g$. The process of finding the $k$ points in $\sF$ yields a transformation from the $k$
evaluations of $g$ to $m$ approximations of $h$; these $m$ approximations of $h$ are then used to compute the
coefficients of a multivariate polynomial surrogate of $h$.
We therefore reduce the cost of the construction from $m$ evaluations of $f$ and $m$ evaluations of $g$ to $m$
evaluations of $f$ and $k \ll m$ evaluations of $g$, plus the cost of finding the $k$ points and computing/applying the
transformation. This is particularly advantageous when $g(f)$ is much more expensive to evaluate than $f(\vx)$.

The $k$ points that we find in $\sF$ are the Gaussian quadrature points from the implicitly approximated density
function of $f$. The process constructs a set of polynomials in $f$ that are orthonormal with respect to the density
function of $f$. The function $g$ is then approximated as a truncated series in these basis polynomials of $f$. 

We use a discrete Stieltjes procedure~\cite{Gautschi04} to compute the recurrence coefficients of the orthogonal
polynomials in $f$, and we show how this is equivalent to a Lanczos' method~\cite{Lanczos50,Meurant06_2} on a diagonal
matrix of the evaluations of $f$ with a properly chosen starting vector. The basis vectors from the Lanczos iteration
are used to linearly map the $k$ evaluations of $g(f)$ to $m$ approximations of $g(f(\vx))$, which are then used to
approximate the coefficients of the multivariate polynomial surrogate of $h$.

A similar measure transformation approach was used in~\cite{Gerritsma2010,Poette2012} to approximate functions with
sharp gradients. However, the theoretical and computational advantages gained by exploiting the connection to Lanczos'
method were not explored.

In what follows, we review the polynomial approximation, Gaussian quadrature, and Lanczos' method (Section
\ref{sec:prelim}); define the problem and derive the approximation method (Section \ref{sec:setup}); and demonstrate
its applicability on two numerical examples: (i) an exponential function composed with a rational function and (ii) a
Navier-Stokes model of fluid flow with a scalar input parameter that depends on multiple physical quantities (Section
\ref{sec:examples}).

\section{Preliminaries}
\label{sec:prelim}

In this section, we briefly review Gaussian quadrature and polynomial approximation, as well as Stieltjes' and
Lanczos' methods. This will also serve to set up notation; we will use notation very similar to  
Gautschi~\cite{Gautschi04}.

\subsection{Orthogonal polynomials, Gaussian quadrature, and pseudospectral approximation}
Let $\sT\subset\mathbb{R}$ be equipped with a positive, normalized weight function $\omega:\sT\rightarrow\mathbb{R}_+$
with finite moments; denote an element of $\sT$ by $t$. For functions $u(t)$ and $v(t)$, define the inner product
\begin{equation}
\label{eq:innerproduct}
(u,v) \;=\; \int_\sT u(t)\,v(t)\,\omega(t)\,dt
\end{equation} 
with associated norm $\|u\| = (u,u)^{1/2}$. Let $\{\bar{\pi}_i(t)\}$ be the set of monic polynomials (i.e., leading
coefficient is 1) that are orthogonal with respect to $\omega(t)$,
\begin{equation}
(\bar{\pi}_i,\bar{\pi}_j) = \left\{
\begin{array}{cl}
\|\bar{\pi}_i\|^2, & \mbox{ if $i=j$}\\
0 & \mbox{ otherwise.}
\end{array}
\right.
\end{equation}
The monic orthogonal polynomials satisfy the recurrence relationship 
\begin{equation}
\label{eq:monic3term}
\bar{\pi}_{i+1}(t) = (t-\alpha_i)\bar{\pi}_i(t) - \beta_i \bar{\pi}_{i-1}(t), \quad i=0,1,2,\dots,
\end{equation}
with $\bar{\pi}_{-1}(t)=0$ and $\bar{\pi}_0(t)=1$. The $\alpha_i$ and $\beta_i$ are given by
\begin{align}
\alpha_i &= \frac{(t\bar{\pi}_i,\,\bar{\pi}_i)}{(\bar{\pi}_i,\bar{\pi}_i)}, \quad i=0,1,2,\dots,
\label{eq:alpha}\\
\beta_i &= \frac{(\bar{\pi}_i,\,\bar{\pi}_i)}{(\bar{\pi}_{i-1},\bar{\pi}_{i-1})}, \quad i=1,2,\dots.
\label{eq:beta}
\end{align}
It is often more convenient to work with orthonormal instead of monic orthogonal polynomials, which we write as
\begin{equation}
\pi_i(t) = \frac{\bar{\pi}_i(t)}{\|\bar{\pi}_i\|}.
\end{equation}
The recurrence relationship for the orthonormal polynomials becomes
\begin{equation}
\label{eq:3term}
\sqrt{\beta}_{i+1}\pi_{i+1}(t) = (t-\alpha_i)\pi_i(t) - \sqrt{\beta_i}\pi_{i-1}(t), \qquad i=0,1,2,\dots,
\end{equation}
with $\pi_{-1}(t)=0$ and $\pi_0(t)=1$. If we consider only the first $n$ equations, then 
\begin{equation}
t\pi_i(t) = \sqrt{\beta_i}\pi_{i-1}(t)+\alpha_i\pi_i(t)+\sqrt{\beta_{i+1}}\pi_{i+1}(t),\qquad i=0,1,\dots,n-1.
\end{equation}
Setting $\bpi(t) = [\pi_0(t),\pi_1(t),\dots,\pi_{n-1}(t)]^T$, we can write this conveniently in matrix form as
\begin{equation}
\label{eq:mat3term}
t\,\bpi(t) = \mJ\,\bpi(t) + \sqrt{\beta_{n}}\,\pi_{n}(t)\,\ve_{n},
\end{equation}
where $\ve_n$ is a vector of zeros with a one in the last entry, and $\mJ=\mJ_n$ (known as the \emph{Jacobi matrix}) is
an $n\times n$ symmetric, tridiagonal matrix
\begin{equation}
\label{eq:jacobi}
\mJ = 
\begin{bmatrix}
\alpha_0 & \sqrt{\beta_1}& & &  \\
\sqrt{\beta_1} & \alpha_1 & \sqrt{\beta_2} & &  \\
 & \ddots & \ddots & \ddots &  \\
 & & \sqrt{\beta_{n-2}} & \alpha_{n-2} & \sqrt{\beta_{n-1}}\\
 & & & \sqrt{\beta_{n-1}} & \alpha_{n-1}
\end{bmatrix}.
\end{equation}
The zeros $\{\lambda_j\}$ of $\pi_{n}(t)$ are the eigenvalues of $\mJ$ and $\bpi(\lambda_j)$ are the corresponding
eigenvectors; this follows directly from \eqref{eq:mat3term}. Let $\mQ$ be the orthogonal matrix of eigenvectors of
$\mJ$; the elements of $\mQ$ are given by
\begin{equation}
\label{eq:eigvJ}
\mQ(i,j) = \frac{\pi_i(\lambda_j)}{\|\bpi(\lambda_j)\|_2},\qquad i,j=0,\dots,n-1,
\end{equation}
where $\|\cdot\|_2$ is the standard 2-norm on $\mathbb{R}^n$. We write the eigenvalue decomposition of $\mJ$ as
\begin{equation}
\label{eq:eigJ}
\mJ = \mQ\Lambda\mQ^T.
\end{equation}
It is known that the eigenvalues $\{\lambda_j\}$ are the nodes of the $n$-point Gaussian quadrature rule associated
with the weight function $\omega(t)$. The quadrature weight $\nu_j$ corresponding to $\lambda_j$ is equal to the square
of the first component of the eigenvector associated with $\lambda_j$, 
\begin{equation}
\label{eq:qweights}
\nu_j \;=\; \mQ(0,j)^2 \;=\; \frac{1}{\|\bpi(\lambda_j)\|_2^2}.
\end{equation}
The weights $\{\nu_j\}$ are known to be strictly positive. 
It will be notationally convenient to define the matrix $\mW=\mathrm{diag}([\sqrt{\nu_0},\dots,\sqrt{\nu_{n-1}}])$.

For an integrable scalar function $f(t)$, we can approximate
its integral by an $n$-point Gaussian quadrature rule. Let $f_j=f(\lambda_j)$. The quadrature approximation
is a weighted sum of function evaluations,
\begin{equation}
\label{eq:gq}
\int_\sT f(t)\,\omega(t)\,dt = \sum_{j=0}^{n-1} f_j\,\nu_j + R(f).
\end{equation}
If $f(t)$ is a polynomial of degree less than or equal to $2n-1$, then $R(f)=0$; that is to say the \emph{degree of
exactness} of the Gaussian quadrature rule is $2n-1$. 

A square integrable function $f(t)$ admits a mean-squared convergent Fourier series in the orthonormal polynomials. A
pseudospectral approximation of $f(t)$ is constructed by first truncating its Fourier series at $n$ terms and
approximating each Fourier coefficient with a quadrature rule. If we use the $n$-point Gaussian quadrature, then we can write
\begin{equation}
\label{eq:pseudospec}
f(t) \;\approx\; \sum_{i=0}^{n-1} \hat{f}_i\,\pi_i(t),
\end{equation}
where
\begin{equation}
\label{eq:pseudocoeff}
\hat{f}_i \;=\; \sum_{j=0}^{n-1} f_j\,\pi_i(\lambda_j)\,\nu_j.
\end{equation}
Let $\vf=[f_0,\dots,f_{n-1}]^T$ and $\hat{\vf}=[\hat{f}_0,\dots,\hat{f}_{n-1}]^T$. Then using $\mQ$ from
\eqref{eq:eigJ}, \eqref{eq:pseudospec}, and \eqref{eq:pseudocoeff}, we can write
\begin{equation}
\label{eq:matpseudo}
\hat{\vf} \;=\; \mQ\mW\vf
\qquad
f(t) \;\approx\; \hat{\vf}^T\bpi(t) \;=\; \vf^T\mW^T\mQ^T\bpi(t).
\end{equation}
The expression $\hat{\vf}=\mQ\mW\vf$ is the discrete Fourier transform for the orthogonal polynomial basis.
Note that it is easy to show that the pseudospectral approximation interpolates $f(t)$ at the Gaussian quadrature
points.

\subsubsection{Tensor product extensions}
\label{sec:tpquad}
The above concepts extend to multivariate functions via a tensor product construction. Let
$\sT=\sT_1\otimes\cdots\otimes\sT_d$ be the domain with elements $\vt=(t_1,\dots,t_d)$. We assume that the weight
function $\omega:\sT\rightarrow\mathbb{R}_+$ is separable, i.e. $\omega(\vt)=\omega_1(t_t)\cdots\omega_d(t_d)$, where
each univariate weight function is normalized and positive with finite moments.

Tensor product Gaussian quadrature rules are constructed by taking cross products
of univariate Gaussian quadrature rules: 
\begin{equation}
\label{eq:quadd}
\blambda_{i_1,\dots,i_d}=(\lambda_{i_1},\dots,\lambda_{i_d}),
\end{equation}
where the points $\lambda_{i_r}$ with $i_r=0,\dots,n_r-1$ are the univariate quadrature points for $\omega_r(t_r)$, with
$r=1,\dots,d$. The associated quadrature weights are given by the products
\begin{equation}
\nu_{i_1,\dots,i_d}=\nu_{i_1}\cdots\nu_{i_d}.
\end{equation}
To approximate the integral of $f(\vt)$, compute
\begin{equation}
\int_\sT f(\vt)\,\omega(\vt)\,d\vt
\;\approx\;
\sum_{i_1=0}^{n_1-1}\cdots\sum_{i_d=0}^{n_d-1} f(\blambda_{i_1,\dots,i_d})\, \nu_{i_1,\dots,i_d}.
\end{equation}
The tensor product pseudospectral approximation is given by
\begin{equation}
f(\vt)
\;\approx\;
\sum_{i_1=0}^{n_1-1}\cdots\sum_{i_d=0}^{n_d-1} \hat{f}_{i_1,\dots,i_d}\, \pi_{i_1}(t_1)\cdots\pi_{i_d}(t_d),
\end{equation}
where 
\begin{equation}
\hat{f}_{i_1,\dots,i_d} 
\;=\; 
\sum_{j_1=0}^{n_1-1}\cdots\sum_{j_d=0}^{n_d-1}
f(\blambda_{i_1,\dots,i_d})\,\pi_{i_1}(\lambda_{j_1})\cdots\pi_{i_d}(\lambda_{j_d})
\,\nu_{i_1,\dots,i_d}.
\end{equation}
The matrix notation extends via Kronecker products. Let $\bpi_{n_r}(t_r)$ be the vector of univariate polynomials that
are orthonormal with respect to $\omega_r(t_r)$ for $r=1,\dots,d$. Then the vector
\begin{equation}
\bpi(\vt) \;=\; \bpi_{n_1}(t_1)\otimes\cdots\otimes\bpi_{n_d}(t_d)
\end{equation}
contains multivariate polynomials that are orthonormal with respect to $\omega(\vt)$. Similarly, define the matrices
\begin{equation}
\label{eq:krondft}
\mQ = \mQ_{n_1}\otimes\cdots\otimes\mQ_{n_d}, \qquad
\mW = \mW_{n_1}\otimes\cdots\otimes\mW_{n_d}.
\end{equation}
Then 
\begin{equation}
\label{eq:matpseudod}
\hat{\vf} \;=\; \mQ\mW\vf
\qquad
f(\vt) \;\approx\; \hat{\vf}^T\bpi(\vt) \;=\; \vf^T\mW^T\mQ^T\bpi(\vt),
\end{equation}
where $\hat{\vf}$ is an $(n_1\cdots n_d)$-vector of the tensor product pseudospectral coefficients, $\vf$ is an
$(n_1\cdots n_d)$-vector containing the evaluations of $f$ at the points given by \eqref{eq:quadd}, ordered
appropriately. Again, the expression $\hat{\vf}=\mQ\mW\vf$ is the $d$-dimensional discrete Fourier transform for the
polynomial basis. 

\subsection{Stieltjes' procedure}
Stieltjes proposed a procedure for iteratively constructing a sequence of univariate polynomials that are orthogonal
with respect to a given measure; see~\cite{Gautschi04}. His method exploits the recurrence relationship for the
orthogonal polynomials. 
He observed that with the weighted inner product \eqref{eq:innerproduct}, one may begin with
$\pi_{-1}$ and $\pi_0$, compute $\alpha_0$ and $\beta_1$ from \eqref{eq:alpha} and \eqref{eq:beta}, construct $\pi_1$
from \eqref{eq:monic3term}, compute $\alpha_1$ and $\beta_2$, construct $\pi_2$, and so on. 

A normalized version of Stieltjes' method for computing the orthonormal polynomials and their recurrence coefficients is
given in Algorithm \ref{alg:normstieltjes}. 
The computed $\alpha_i$ from Algorithm \ref{alg:normstieltjes} are equivalent to the expression in \eqref{eq:alpha}, and
the computed $\eta_i$ are equal to $\sqrt{\beta_i}$ in \eqref{eq:beta}.

\begin{algorithm}
\caption{A Stieltjes procedure for computing the first $n$ orthonormal polynomials given a normalized weight function
$\omega(t)$. Let $\pi_{-1}=0$ and $\tilde{\pi}_0=1$.}
\label{alg:normstieltjes}
\begin{algorithmic}
\FOR{$i=0$ to $n-1$}
\STATE $\eta_i=\|\tilde{\pi}_i\|$
\STATE $\pi_i=\tilde{\pi}_i/\eta_i$
\STATE $\alpha_i=(t\pi_i,\pi_i)$
\STATE $\tilde{\pi}_{i+1} = (t-\alpha_i)\pi_i - \eta_i\pi_{i-1}$
\ENDFOR
\end{algorithmic}
\end{algorithm}

Gautschi~\cite{Gautschi04} proposed to use a discrete inner product -- e.g., based on a Gaussian quadrature rule. In the
univariate case ($d=1$), this becomes
\begin{equation}
\label{eq:dinnerprod}
(u,v) \;\approx\; \sum_{j=0}^{m-1} u(\lambda_j)\,v(\lambda_j)\,\nu_j,
\end{equation}
where $\lambda_j$ and $\nu_j$ are the points and weights of the discrete inner product. He reasoned that if the discrete
inner product converges to the continuous, then the recurrence coefficients computed with the discrete inner product
will also converge. Similarly, one may think of the recurrence coefficients computed with the discrete inner product as
approximations of those computed with the continuous inner product.

In Section \ref{sec:setup}, we will use a tensor product quadrature rule to define a discrete 
inner product on the space $\sX$. We will use this inner product in a Stieltjes procedure to compute the recurrence
coefficients of a set of univariate orthogonal polynomials, where orthogonality is with respect to the density function
of $f$ defined on $\sF$. We have chosen to focus the presentation by using the tensor product quadrature rule to define
the discrete inner product. However, the construction can be easily adjusted to employ other discrete inner products.

\subsection{Lanczos method}
Lanczos' method~\cite{Lanczos50} for symmetric matrices is the foundation for iterative eigensolvers and Krylov
subspace methods for solving symmetric linear systems. It generates a symmetric, tridiagonal matrix (the Jacobi
matrix) and a sequence of mutually orthogonal (in exact arithmetic) vectors known as the Lanczos vectors. The
eigenvalues of the tridiagonal matrix -- known as the Ritz values -- approximate the eigenvalues of the symmetric
matrix.

In fact, Algorithm \ref{alg:normstieltjes} is exactly a form of Lanczos' method\footnote{However, Algorithm
\ref{alg:normstieltjes} has undesirable numerical properties as an implementation.}, if we replace (i) the variable $t$
by a symmetric matrix $\mA$ of size $m\times m$, (ii) the polynomials $\pi_i(t)$ by the Lanczos vectors $\vv_i$,
(iii) the starting polynomial $\tilde{\pi}_0$ by a starting vector $\tilde{\vv}_0$, and (iv) the inner product by a
discrete, weighted inner product. 

Suppose that $k$ iterations of the method have been executed. We can write the recurrence relationship for the Lanczos
vectors in matrix notation as
\begin{equation}
\label{eq:lanczosrecur}
\mA\mV \;=\; \mV\mT \,+\, \eta_k\vv_{k}\ve_k^T,
\end{equation}
where $\mV=[\vv_0,\dots,\vv_{k-1}]$ is an $m\times k$ matrix of Lanczos vectors, $\mT$ is the $k\times k$ symmetric,
tridiagonal Jacobi matrix of recurrence coefficients, and $\ve_k$ is a last column of the $k\times k$ identity matrix.
\section{The approximation method}
\label{sec:setup}

In this section, we take advantage of the relationship between an approximate Stietjes' procedure with a discrete inner
product and Lanczos' method to devise a computational method for approximating composite functions. Recall the
problem setup from the introduction:
\begin{equation}
\label{eq:hdef2}
h(\vx)=g(f(\vx)), \qquad \vx=(x_1,\dots,x_d)\in\sX\subset\mathbb{R}^d
\end{equation} 
where
\begin{align}
f&:\sX \;\longrightarrow\; \sF\; \subset\; \mathbb{R}\\
g&:\sF \;\longrightarrow\; \sG\; \subset\; \mathbb{R}.
\end{align}
We restrict our attention to input spaces that are $d$-dimensional hypercubes, $\sX=[-1,1]^d$, although this can be
relaxed to more general hyperrectangles. We assume that $\sX$ is equipped with a positive, separable weight function
$\omega_{\vx}=\omega(\vx)=\omega_1(x_1)\cdots\omega_d(x_d)$ with finite moments. We also assume that $f$ is bounded, so
that $\sF$ is a closed interval in $\mathbb{R}$.

A tensor product pseudospectral approximation of $h$ on the space $\sX$ follows the construction in Section
\ref{sec:tpquad}:
\begin{equation}
\label{eq:pseudospech}
h(\vx) 
\;\approx\;
\sum_{i_1=0}^{n_1-1}\cdots\sum_{i_d=0}^{n_d-1} \hat{h}_{i_1,\dots,i_d}\, \pi_{i_1}(x_1)\cdots\pi_{i_d}(x_d),
\end{equation}
where the $\pi_{i_r}(x_r)$ are the univariate polynomials that are orthogonal with respect to $\omega_r(x_r)$. The
coefficients $\hat{h}_{i_1,\dots,i_d}$ are 
\begin{equation}
\hat{h}_{i_1,\dots,i_d} 
\;=\; 
\sum_{j_1=0}^{n_1-1}\cdots\sum_{j_d=0}^{n_d-1}
h(\blambda_{j_1,\dots,j_d})\,\pi_{i_1}(\lambda_{j_1})\cdots\pi_{i_d}(\lambda_{j_d})
\,\nu_{j_1,\cdots,j_d},
\end{equation}
where $\blambda_{j_1,\dots,j_d}=(\lambda_{j_1},\dots,\lambda_{j_d})$ and
$\nu_{j_1,\cdots,j_d}=\nu_{j_1}\cdots\nu_{j_d}$ are the nodes and weights, respectively, of the tensor product Gaussian
quadrature for the weight function $\omega_{\vx}$. This can be written conveniently in matrix form as in
\eqref{eq:matpseudod} as
\begin{equation}
\label{eq:pseudoh}
\hat{\vh} \;=\; \mQ_{\vx}\mW_{\vx}\vh,
\qquad
h(\vx) \;\approx\; \hat{\vh}^T\bpi(\vx) \;=\; \vh^T\mW_{\vx}^T\mQ_{\vx}^T\bpi(\vx),
\end{equation}
where $\hat{\vh}$ is a vector of the tensor product pseudospectral coefficients; $\vh$ is a vector of evaluations of
$h$ at the quadrature points; $\bpi(\vx)$ is a vector of the product type multivariate orthonormal polynomials; and 
$\mW_{\vx}$ and $\mQ_{\vx}$ represent the $d$-dimensional discrete Fourier transform for functions defined on $\sX$. 

We assume that the primary expense of this computation comes from the evaluations of $h$ at the quadrature points.
There are $m=n_1\cdots n_d$ points in the tensor product quadrature rule; if $n_1=\cdots=n_d=n$, then $m=n^d$. Each
evaluation of $h$ requires an evaluation of $f$ followed by an evaluation of $g$. Our goal is to approximate the values
of $h$ at the quadrature points using $m$ evaluations of $f$, but only $k \ll m$ evaluations of $g$; we accomplish this
goal by taking advantage of the composite structure of $h$. 

Our strategy is to apply Stieltjes'
procedure to implicitly construct a set of polynomials that are orthonormal with respect to the density function of $f$
defined on $\sF$. The $k$-point Gaussian quadrature rule for the univariate density function 
provides a set of points in $\sF$ on which to evaluate $g$. We then linearly map the $k$ evaluations of $g$ to
approximate the $m$ evaluations of $h$. 

Let $\omega_f:\sF\rightarrow \mathbb{R}_+$ be the normalized density function of $f$. We assumed $f$ was bounded on
$\sX$, which implies $\sF$ is a closed interval and $\omega_f$ is bounded. Therefore, $\omega_f$ has finite moments, and
it admits a set of univariate orthonormal polynomials $\phi_i=\phi_i(f)$ with $i=0,1,2,\dots$. Then we can approximate
$g=g(f)$ with a univariate pseudospectral approximation, 
\begin{equation}
g(f) \;\approx\; \sum_{i=0}^{k-1} \hat{g}_i\, \phi_i(f),
\end{equation}
where 
\begin{equation}
\label{eq:hatg}
\hat{g}_i \;=\; \sum_{j=0}^{k-1} g(\theta_j)\,\phi_i(\theta_j)\,\mu_j.
\end{equation}
The $\theta_j\in\sF$ and $\mu_j\in\mathbb{R}_+$ are the nodes and weights, respectively, of the $k$-point Gaussian
quadrature rule for $\omega_f$. In the matrix notation, 
\begin{equation}
\label{eq:mathatg}
\hat{\vg} \;=\; \mQ_f\mW_f\vg,
\qquad
g(f) \;\approx\; \hat{\vg}^T\bphi(f) \;=\; \vg^T\mW_f^T\mQ_f^T\bphi(f),
\end{equation}
where $\hat{\vg}$ is a vector of the pseudospectral coefficients; $\vg$ is a vector of evaluations of
$g$ at the quadrature points $\theta_j$; $\bphi(f)$ is a vector of the univariate orthonormal polynomials; and
$\mW_{f}$ and $\mQ_{f}$ represent the discrete Fourier transform for functions defined on $\sF$.

Notice that if we were given (i) the coefficients $\hat{g}_i$ and (ii) the evaluations of the polynomials $\phi_i$ at
the point $f(\vx)$ with $\vx\in\sX$, then 
\begin{equation}
\label{eq:pseudospecg}
h(\vx) \;=\; g(f(\vx)) \;\approx\; \sum_{i=0}^{k-1} \hat{g}_i\, \phi_i(f(\vx)).
\end{equation}
Unfortunately, the direct evaluation of this expression is problematic, since each new point $\vx$ requires the
computation of $f(\vx)$. We would prefer a surrogate like \eqref{eq:pseudospech} whose evaluation is independent of
both $g$ and $f$. 
With this in mind, we seek to evaluate \eqref{eq:pseudospecg} at only the quadrature points $\blambda_{i_1,\dots,i_d}$
on $\sX$:
\begin{equation}
\label{eq:htilde}
h(\blambda_{j_1,\dots,j_d}) \;\approx\; \underline{h}(\blambda_{j_1,\dots,j_d}) \;=\; 
\sum_{i=0}^{k-1} \hat{g}_i\,\phi_i(f(\blambda_{j_1,\dots,j_d})).
\end{equation}
We denote the output of the pseudospectral expansion \eqref{eq:htilde} by $\underline{h}$. We write \eqref{eq:htilde} in
matrix notation as
\begin{equation}
\label{eq:mathtilde}
\vh \;\approx\; \underline{\vh} \;=\; \mU\hat{\vg}.
\end{equation}
The elements of the $m\times k$ matrix $\mU$ are the evaluations of $\phi_i(f(\blambda_{i_1,\dots,i_d}))$,
where each row of $\mU$ corresponds to a quadrature point $\blambda_{i_1,\dots,i_d}$ and each column
corresponds to a polynomial $\phi_i(f)$. This matrix enables a construction similar to \eqref{eq:pseudospech}.

Putting these ideas together, we construct a global polynomial surrogate for $h(\vx)$ that approximates the
tensor product pseudospectral approximation. The form of the polynomial surrogate is the same as the true pseudospectral
approximation -- i.e., a linear combination of the multivariate orthogonal basis polynomials -- but the
pseudospectral coefficients are approximated using $k\ll m$ evaluations of $g$ plus the cost of computing the polynomial
evaluations $\phi_i(f(\blambda_{i_1,\dots,i_d}))$. We will denote the approximate coefficients by
$\doublehat{h}\approx\hat{h}$, which is consistent with the notation $\underline{h}\approx h$. 

More precisely, we construct a surrogate for $h(\vx)$ as
\begin{equation}
h(\vx) 
\;\approx\;
\sum_{i_1=0}^{n_1-1}\cdots\sum_{i_d=0}^{n_d-1} \doublehat{h}_{i_1,\dots,i_d}\, \pi_{i_1}(x_1)\cdots\pi_{i_d}(x_d),
\end{equation}
where
\begin{equation}
\doublehat{h}_{i_1,\dots,i_d} 
\;=\; 
\sum_{j_1=0}^{n_1-1}\cdots\sum_{j_d=0}^{n_d-1}
\underline{h}(\blambda_{j_1,\dots,j_d})\,\pi_{i_1}(\lambda_{j_1})\cdots\pi_{i_d}(\lambda_{j_d})
\,\nu_{j_1,\cdots,j_d}.
\end{equation}
The quantities $\underline{h}(\blambda_{j_1,\dots,j_d})$ are given by \eqref{eq:htilde}. In matrix notation, we combine
\eqref{eq:pseudoh}, \eqref{eq:mathatg}, and \eqref{eq:mathtilde} to get
\begin{align}
h(\vx) &\approx \doublehat{\vh}^T\bpi(\vx)\\
 &= \underline{\vh}^T\mW_{\vx}^T\mQ_{\vx}^T\bpi(\vx)\\
 &= \hat{\vg}^T\mU^T\mW_{\vx}^T\mQ_{\vx}^T\bpi(\vx)\\
 &= \vg^T\mW_f^T\mQ_f^T\mU^T\mW_{\vx}^T\mQ_{\vx}^T\bpi(\vx)\label{eq:billy}, 
\end{align}
where $\doublehat{\vh}$ is a vector of the coefficients $\doublehat{h}_{i_1,\dots,i_d}$, and $\underline{\vh}$ is a vector
of the evaluations of $\underline{h}$.  Notice that $\vg$ contains $k$ evaluations of $g$. In what follows, we will show
how to compute the points $\theta_j$ in \eqref{eq:hatg} and the transformations $\mU$, $\mQ_f$, and
$\mW_f$ from \eqref{eq:billy} using $m$ evaluations $f(\blambda_{j_1,\dots,j_d})$. Thus, we will approximate the tensor
product pseudospectral approximation \eqref{eq:pseudospech} using $k$ evaluations of $g$ and $m$ evaluations of $f$.

\subsection{Lanczos' method for approximation}
The strategy is implemented computationally with Lanczos' method applied to a diagonal matrix whose nonzero
elements are the $m$ evaluations of $f$ at the quadrature points on $\sX$; the $k$-point Gaussian quadrature rule on
$\sF$ comes from the computed Jacobi matrix (see \eqref{eq:eigJ}), and the map from $k$ evaluations of $g$ to $m$
evaluations of $h$ comes from the Lanczos vectors.

It is notationally convenient in this section to order the $d$-dimensional quadrature points
$\blambda_{i_1,\dots,i_d}\in\sX$ so as to be indexed by the natural numbers. For the remainder, we will refer to a
node as $\blambda_i$ with corresponding weight $\nu_i$ for $i=0,\dots,m-1$. The specific ordering must be consistent
with the tensor product structure of \eqref{eq:krondft}. We will similarly order the function evaluations
$f_i=f(\blambda_i)$.

We first show that Lanczos' method yields the quantities desired from the Stieltjes procedure. The fact has been
observed elsewhere in the literature~\cite{Gragg84,Greenbaum97,Oleary07,Golub10}; we state it as a theorem for reference
and notation.

\begin{theorem}
\label{thm:lanczos}
Let
\begin{equation}
\mA=\bmat{f_0 & & \\ & \ddots & \\ & & f_{m-1}}
\end{equation}
be the diagonal matrix whose nonzero elements are the evaluations of $f$ at the quadrature nodes
$\blambda_i\in\sX$. For an $m$-vector of ones $\ve$ and a starting vector $\tilde{\vv}_0=\mW_{\vx}\ve$,
Lanczos method applied to $\mA$ is equivalent to Stieltjes' procedure with a discrete inner product to construct the 
recurrence coefficients of the polynomials $\phi_i(f)$ that are orthonormal with respect to an approximation of the
measure $\omega_f$. The discrete inner product is defined by the nodes $\blambda_i$ and weights $\nu_i$ of the tensor
product Gaussian quadrature rule for the measure $\omega_{\vx}$.
\end{theorem}

\begin{proof}
To prove this statement, we simply describe the quantities in Algorithm \ref{alg:normstieltjes} with the discrete inner
product. The starting polynomial $\tilde{\phi}_0=1$ corresponds to an $m$-vector of ones, $\ve$. Let $\vv_{-1}$ be an
$m$-vector of zeros. Then the quantities from Algorithm \ref{alg:normstieltjes} with the discrete inner product become
\begin{align*}
\eta_i &= \left(\sum_{j=0}^{m-1} \tilde{\phi}_i(f_j)^2\, \nu_j\right)^{1/2}
\;=\; \left(\tilde{\vv}_i^T\tilde{\vv}_i\right)^{1/2}\\
\sqrt{\nu_j}\phi_i(f_j) &= \frac{\tilde{\phi}_i(f_j)}{\eta_i},
\quad
\vv_i=[\sqrt{\nu_0}\phi_i(f_0),\dots, \sqrt{\nu_{m-1}}\phi_i(f_{m-1})]^T,\\
\alpha_i &= \sum_{j=0}^{m-1} f_j\,\phi_i(f_j)^2 \nu_j
\;=\; \vv_i^T\mA\vv_i,\\
\tilde{\phi}_{i+1}(f_j) &= (f_j\,-\,\alpha_i)\,\phi_i(f_j) - \eta_i\,\phi_{i-1}(f_j),
\end{align*}
which can be written
\begin{equation}
\tilde{\vv}_{i+1} = (\mA-\alpha_i\mI)\vv_i - \eta_i\vv_{i-1},
\end{equation}
with $\tilde{\vv}_{i+1}=[\tilde{\phi}_{i+1}(f_0),\dots,\tilde{\phi}_{i+1}(f_{m-1})]^T$. To recover the
polynomials, 
\begin{equation}
\label{eq:scaledpolys}
\mU=\mW_{\vx}^{-1}\mV,
\end{equation}
where $\mU(j,i)=\phi_i(f_j)$.  
\end{proof}

The matrix $\mU$ from \eqref{eq:scaledpolys} is the same as the one from \eqref{eq:mathtilde}. Note that -- as mentioned
in Theorem \ref{thm:lanczos} -- these quantities do not correspond to the exact measure $\omega_f$. They instead
correspond to an approximation of $\omega_f$ from the evaluations $f_j=f(\blambda_j)$. We will not examine the error
made in this approximation; we assume the evaluations of $f$ at the nodes $\blambda_j\in\sX$ are sufficient to resolve
the salient features of $\omega_f$.

Running $k$ steps of the Lanczos process yields the recurrence relationship \eqref{eq:lanczosrecur}.
The elements of the $k\times k$ tridiagonal Jacobi matrix $\mT$ are the recurrence coefficients of the polynomials
$\phi_i(f)$ up to order $k-1$. The Lanczos vectors $\mV$ contain the polynomial evaluations $\phi_i(f_j)$ scaled by
$\sqrt{\nu_j}$ as in \eqref{eq:scaledpolys}.

Denote the eigendecomposition of $\mT$ by
\begin{equation}
\label{eq:eigT}
\mT = \mQ_f\Theta\,\mQ_f^T, \qquad \Theta=\mathrm{diag}([\theta_0,\dots,\theta_{k-1}]).
\end{equation}
The Ritz values (the eigenvalues of $\mT$) are the Gaussian quadrature nodes for the approximation of $\omega_f$ on the 
space $\sF$, and the weights come from the first component of the eigenvectors of $\mT$ as in \eqref{eq:qweights}.
Precisely speaking, the $\phi_i(f)$ are orthogonal with respect to the discrete measure defined by the $\theta_j$ and
$\mu_j$.

Therefore, we compute the quantities $\theta_j$, $\mu_j$, $\mQ_f$, $\mW_f$, and $\mU$ from \eqref{eq:billy} using $k$
steps of Lanczos' method applied to the diagonal matrix $\mA$ followed by the eigendecomposition of the Jacobi matrix
$\mT$. This is the computational cost incurred beyond the $m$ evaluations of $f$ and $k$ evaluations of $g$.

\subsection{Loss of orthogonality and stopping criteria}
We have stated that we expect $k\ll m$, or that the number of points in the discrete measure on $\sF$ will be much
smaller than the number of points in the discrete measure on $\sX$. The number $k$ is the number of iterations of the
Lanczos procedure; how do we know how many iterations to use to get an accurate approximation of the measure on $\sF$? 

It is well known that Lanczos' method in finite precision behaves differently than the algorithm in exact arithmetic; a
thorough treatment of this subject can be found in Meurant's excellent monograph~\cite{Meurant06_2}, as well
as~\cite{Meurant06}.
In particular, the Lanczos vectors will often lose orthogonality after some number of iterations.

Thanks to the work of Paige~\cite{Paige71} as described in~\cite{Meurant06_2} -- as well
as~\cite{Parlett98,Greenbaum89,Greenbaum92} -- we know that the loss of orthogonality is closely related to the
convergence of the Ritz values to the true eigenvalues; loosely speaking, once a Ritz value has converged to an
eigenvalue, the remaining Lanczos vectors lose orthogonality. It has been observed that in many cases the extremal Ritz
values converge to the extremal eigenvalues fastest depending on the starting vector. From this we can expect that the
Lanczos vectors will lose orthogonality once the extremal Ritz values are sufficiently close to the extremal
eigenvalues. We use this expectation to motivate a heuristic for stopping the Lanczos iteration. Further justification
of the following heuristic is the subject of on-going work.

Recall that $\mA$ is diagonal, so we are not concerned with any particular eigenvalue. In fact, we are only concerned
with approximating the range of the data -- which is the range of the function $f(\vx)$ evaluated at the points
$\blambda_i$ -- and its corresponding measure. Therefore, once the extremal Ritz values converge, we are satisfied.
Leveraging the work on Lanczos' method in finite precision, we can judge when the extremal Ritz values have converged
by checking orthogonality of the Lanczos vectors. Essentially, we can treat the loss of orthogonality in the Lanczos
vectors as stopping criteria. We use the following measure of loss of orthogonality given a tolerance TOL:
\begin{equation}
\label{eq:stop}
\tau\;=\;\log_{10}\big( \| \mI - \mV^T\mV \|_F \big)\;>\;\mathrm{TOL}
\end{equation}
where $\|\cdot\|_F$ is the Frobenius norm. Other efficient measures for determining loss of orthogonality are discussed
in~\cite[Chapter 9]{GVL96} as well as~\cite{Simon84,Simon1984_2}. In the numerical examples of Section
\ref{sec:examples}, we choose TOL=-14.

If the iterations continue beyond this point, we find that the points and weights of the quadrature rule for the measure
on $\sF$ become less smooth; this phenomenon is similar to choosing the wrong bin size for a histogram. In some cases,
we observe the familiar (to those who have studied Lanczos' method) appearance of ghost eigenvalues. If we examine
the weights corresponding to pairs of nearly identical Ritz values, we usually find that one of the weights is
orders of magnitude smaller than the other. Of course, we would prefer to ignore points with very small weights, since
this would correspond to a wasted function evaluation in the quadrature approximations. We demonstrate these phenomena
on the following numerical examples.

\subsection{An Algorithm}
We close this section with a step-by-step algorithm using the linear algebra notation to summarize the procedure. Given
functions $f(\vx)$ and $g(f)$, the goal is to approximate the coefficients of a tensor product pseudospectral
polynomial surrogate for the composite function $h(\vx)=g(f(\vx))$.  
\begin{enumerate}
  \item Obtain the $m$ nodes $\blambda_i$ and weights $\nu_i$ of the tensor product Gaussian quadrature rule for the
  space $\sX$. Also, obtain the $m\times m$ matrix $\mQ_{\vx}$ and the diagonal matrix of the square root of the
  quadrature weights $\mW_{\vx}$ from \eqref{eq:matpseudod}\footnote{In a real implementation, $\mQ_{\vx}$ and
  $\mW_{\vx}$ do not need to be formed explicitly. We only need the action of $\mQ_{\vx}$ on a vector, which can be
  computed efficiently using methods such as~\cite{Fernandes98}.}.
  \item For $i=0,\dots,m-1$, compute $f_i=f(\blambda_i)$, and form the diagonal matrix
  $\mA=\mathrm{diag}\,([f_0,\dots,f_{m-1}])$.
  \item For the starting vector $\tilde{\vv}_0=\mW_{\vx}\ve$, run Lanczos' method until $\tau>\mathrm{TOL}$ from
  \eqref{eq:stop}, and let $k$ be the number of iterations. Store the matrix $\mU$ from \eqref{eq:scaledpolys}.
  \item Compute the eigenvalue decomposition of the $k\times k$ Jacobi matrix $\mT$ from the Lanczos procedure to get
  the quadrature nodes $\{\theta_k\}$ and the discrete Fourier transform matrices $\mQ_f$ and $\mW_f$; see
  \eqref{eq:eigT}.
  \item For $i=0,\dots,k-1$, compute $g_i=g(\theta_i)$ and form the vector $\vg=[g_0,\dots,g_{k-1}]^T$.
  \item Compute approximate coefficients for the pseudospectral approximation $\doublehat{\vh}\approx\hat{\vh}$ as
  \begin{equation}
  \doublehat{\vh} \;=\; \mQ_{\vx}\mW_{\vx}\mU\,\mQ_f\mW_f\vg. 
  \end{equation}
\end{enumerate}
These coefficients define a polynomial approximation of $h(\vx)$ with a basis of multivariate product-type orthonormal
polynomials.

\section{Numerical Examples}
\label{sec:examples}

We present two numerical studies demonstrating the qualities of the method. The first is an example with 
functions chosen to stress the method's properties. The second applies the method -- as a proof of concept -- to a
model from fluid dynamics with a scalar input parameter that depends on multiple physical quantities.

\subsection{Simple functions}
Let $\vx=(x_1,x_2)\in [-1,1]^2$ with a uniform measure of $1/4$ in $[-1,1]^2$ and zero otherwise.
Given parameters $\delta_1>1$ and $\delta_2>1$, define the function
\begin{equation}
f(\vx) = \frac{1}{(x_1-\delta_1)(x_2-\delta_2)}.
\end{equation}
Notice that $f(\vx)>0$, and $\delta_1$ and $\delta_2$ determine how quickly $f$ grows near the boundary. The closer
$\delta_1$ and $\delta_2$ are to 1, the closer the singularity in the function gets to the domain, which determines how
large $f$ is at the point $(x_1=1,x_2=1)$. For the numerical experiments, we choose $\delta_1=\delta_2=1.3$. The
function $f$ is analytic in $\vx$, so we expect polynomial approximations to converge exponentially as the degree of
approximation increases.

Next we choose $g(f)=\exp(f)$, so that
\begin{equation}
h(\vx) \;=\; g(f(\vx)) \;=\; \exp\left(\frac{1}{(x_1-\delta_1)(x_2-\delta_2)}\right).
\end{equation}
Again, $g(f)$ is analytic in $f$, so $h(\vx)$ is analytic in $\vx$, as well. 

We choose the discrete measure on $\sX$ to be a tensor product Gauss-Legendre quadrature rule on $[-1,1]^2$ with $n$
points in each variable, which results in $m=n^2$ points and weights. The $m\times m$ diagonal matrix $\mA$ has
diagonal elements equal to $f$ evaluated at the points of the discrete measure. 

To test the approximation properties of the Lanczos-enabled method, it is sufficient to examine the error at the tensor
product quadrature points; see \eqref{eq:htilde}. This is essentially the same as computing the difference between the
true pseudospectral coefficients $\hat{h}$ and their approximation $\underline{\hat{h}}$. 
For the number $m$ of bivariate nodes of the tensor product quadrature rule and the number $k$ of Lanczos
iterations, define the error $\sE=\sE(m,k)$ as
\begin{align}
\label{eq:lerror}
\sE^2 &= \sum_{i=0}^{m-1} \big( h(\blambda_i) - \sum_{j=0}^{k-1} \hat{g}_j\,\phi_j(f(\blambda_i)) \big)^2 \\  
 &= \|\vh - \mU\hat{\vg}\|_2^2.
\end{align}
In Figure \ref{fig:dmeas}, we plot both $\log_{10}(\sE)$ and the measure of orthogonality of the Lanczos vectors
(see \eqref{eq:stop}) as $m$ and $k$ increase. To read these plots, choose $m$ from the $y$-axis, and then
follow the plot to the right to increase the Lanczos iteration $k$. It is interesting to note that the error in the
approximation does not increase after the Lanczos vectors lose orthogonality. The loss of orthogonality is useful for a
stopping criteria to determine the smallest $k$ that produces an accurate approximation. However, taking more than the
minimum number of Lanczos iterations does no harm for this example.

\begin{figure}%
 \subfloat[Error]{\label{fa}
 \includegraphics[scale=0.35]{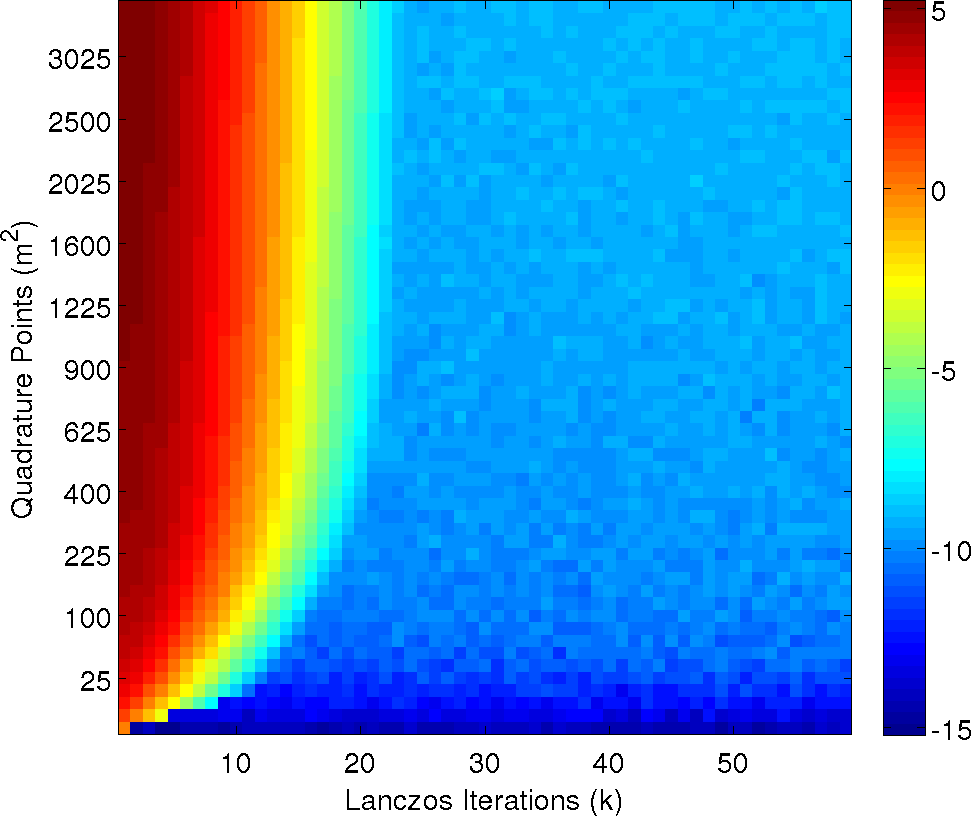}
 }
 \quad
 \subfloat[Orthogonality]{\label{fb}
 \includegraphics[scale=0.35]{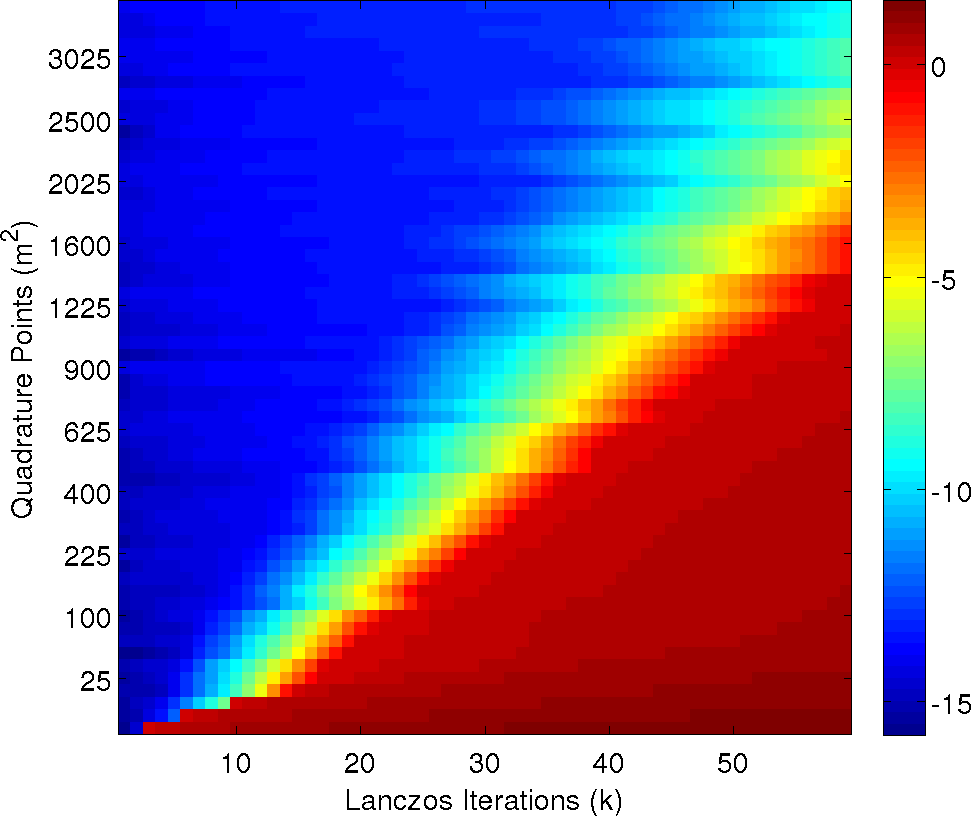}
 }
 \caption{Figure~\ref{fa} plots the error in approximation of $h$ as measured by \eqref{eq:lerror}. Figure~\ref{fb}
 shows the loss of orthogonality in the Lanczos vectors using $\tau$ from \eqref{eq:stop}.} 
 \label{fig:dmeas}
\end{figure}

In Figure \ref{fig:bv1}, we plot a series of bar graphs of the quadrature weights $\mu_l$ at points $\theta_l$ for the
measure on $\sF$ computed with $m=81$.
While the bar plot resembles a histogram, the comparison between a histogram and quadrature weights
is not precise. Nevertheless, the series of bar plots demonstrates the behavior of the weights as the Lanczos iteration
index continues beyond the point when the Lanczos vectors lose orthogonality; the orthogonality measures from
\eqref{eq:stop} are presented in each plot. We observe that the weights lose smoothness as the Lanczos vectors lose
orthogonality; note the weights in the right tail of the plot. 

\begin{figure}%
\centering
\subfloat[$k=5$]{
\includegraphics[scale=0.35]{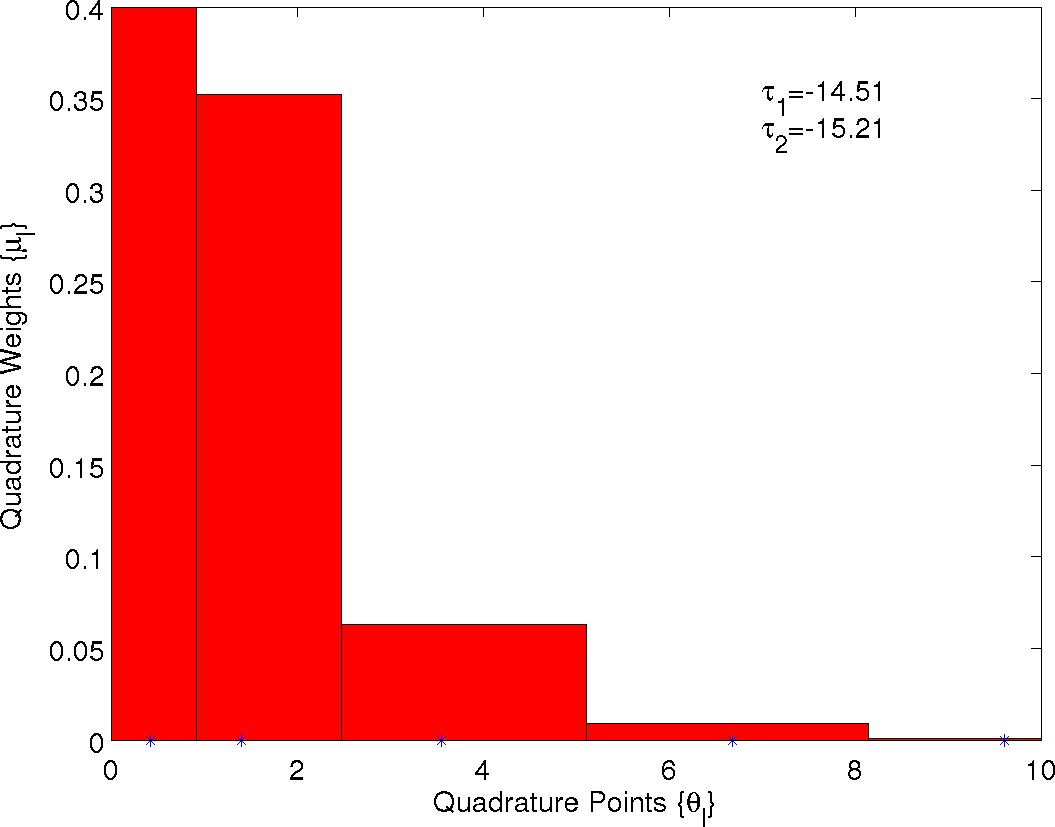}
}
\quad
\subfloat[$k=10$]{
\includegraphics[scale=0.35]{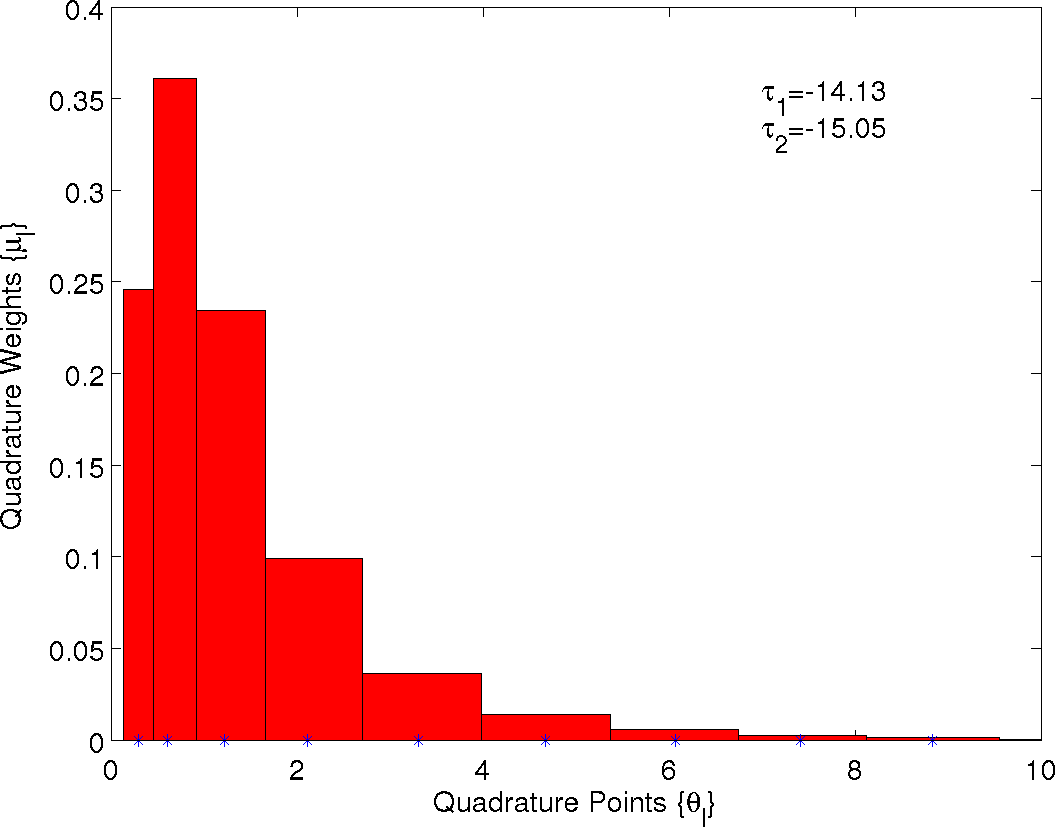}
}\\
\subfloat[$k=15$]{
\includegraphics[scale=0.35]{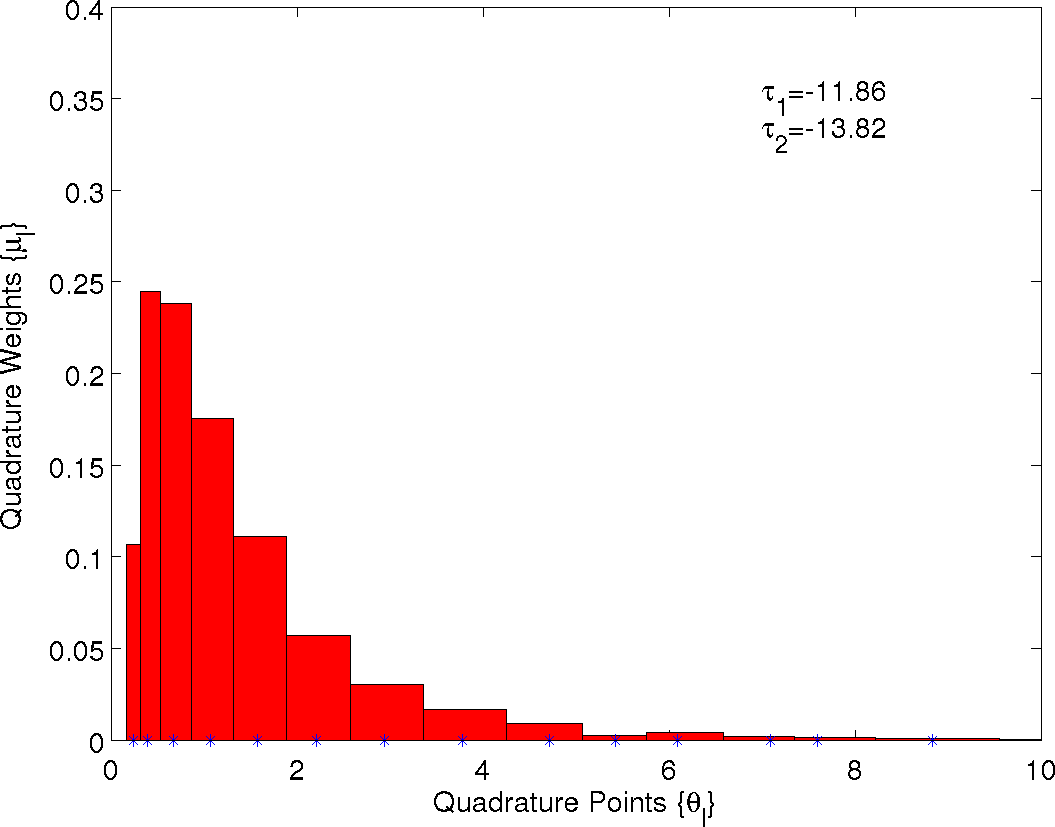}
}
\quad
\subfloat[$k=20$]{
\includegraphics[scale=0.35]{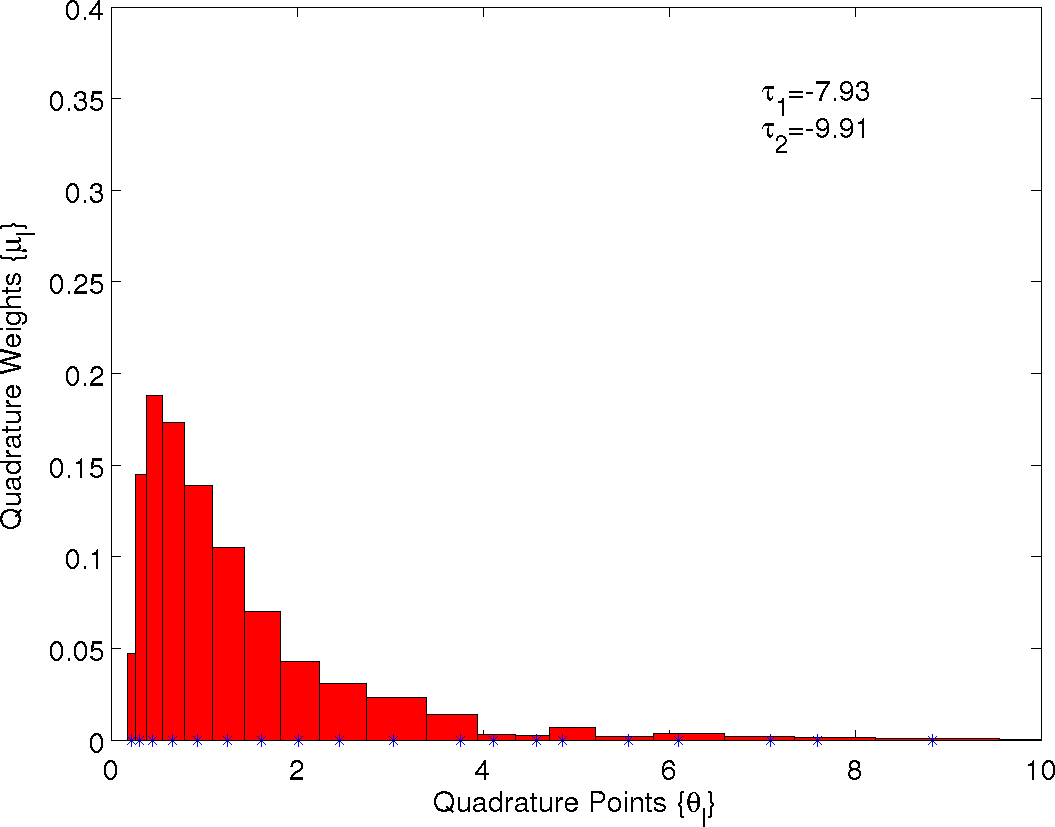}
}
\caption{A series of bar plots showing the weights $\mu_l$ at the points $\theta_l$ for the quadrature rule on $\sF$.
The number $\tau$ in each plot shows the measure of orthogonality in the Lanczos vectors; it is
defined in \eqref{eq:stop}.}
\label{fig:bv1}
\end{figure}

\subsection{Fluid flow example}

As an example of applying these techniques to an engineering problem, we examine a
simple channel flow problem with a scalar input parameter (the Reynolds number) that depends on multiple physical
quantities (density and viscosity). Consider the two-dimensional rectangular domain of length $L = 1$ m and width $W =
0.1$ m shown in Figure~\ref{fig:flow_domain}. Water flows into the left side of the domain with a horizontal velocity
of $u_0=0.01$ m/sec, and we are interested in computing the velocity of the flow out of the domain on the right side.
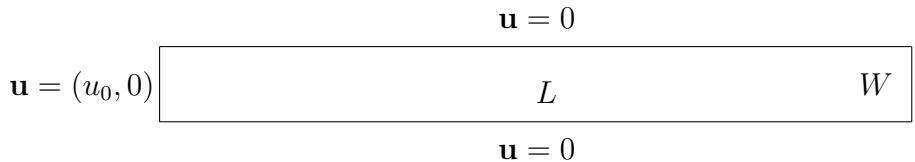
\begin{figure}
\begin{center}
\setlength{\unitlength}{1cm}
\begin{picture}(10,5)
\put(0,0.5){\line(1,0){10}}
\put(10,0.5){\line(0,1){1}}
\put(0,0.5){\line(0,1){1}}
\put(0,1.5){\line(1,0){10}}
\put(4.5,0){$\vu=0$}
\put(4.5,1.75){$\vu=0$}
\put(-2.0,0.85){$\vu = (u_0,0)$}
\put(9.3,0.85){$W$}
\put(5,0.75){$L$}
\end{picture}
\end{center}
\caption{Fluid flow domain}\label{fig:flow_domain}
\end{figure}

At room temperature and standard pressure, the dynamics of the fluid within the domain are well-modeled by the
incompressible Navier-Stokes equations
\begin{align}
  \rho\frac{\partial\vu}{\partial t} + \rho(\vu\cdot\nabla\vu) - \mu\nabla^2\vu + \nabla P &= 0, \\
  \rho(\nabla\cdot\vu) &= 0,
\end{align}
where $\vu$ is the two-component velocity of the fluid, $\rho$ is the density, $\mu$ is the viscosity, and $P$ is the
pressure.  Using the inlet flow velocity and the width $W$ of the domain, the equations are non-dimensionalized
resulting in
\begin{align}
  \frac{\partial\bar{\vu}}{\partial\bar{t}} + \bar{\vu}\cdot\bar{\nabla}\bar{\vu} - \frac{1}{Re}\bar{\nabla}^2\vu +
  \bar{\nabla}\bar{P} &= 0,\label{eq:ns_mom_nondim} \\
  \bar{\nabla}\cdot\bar{\vu} &= 0,\label{eq:ns_cont_nondim}
\end{align}
where 
$\bar{\nabla} = \nabla/W$, and
\begin{equation}
  Re = \frac{\rho u_0 W}{\mu}
\end{equation}
is the Reynolds number.  

Equations~\ref{eq:ns_mom_nondim}-\ref{eq:ns_cont_nondim} are discretized spatially on a mesh of 500 by 50 quadrangle
cells using the finite element method with piecewise bilinear basis functions for both the velocities and
pressures~\cite{Shadid:2006p7554}. Given a Reynolds number, the resulting nonlinear algebraic equations are solved via
Newton's method using a GMRES linear solver~\cite{Saad1996} and incomplete-LU factorization preconditioner.  The
resulting flow solution at density $\rho=\rho_0=998.205 \;\mathrm{km/m}^3$ and viscosity
$\mu=\mu_0=0.001001\;\mathrm{Ns/m}^2$ is shown in Figure~\ref{fig:flow_solution}; the density and viscosity values
roughly correspond to water at room temperature and standard pressure. The calculations were implemented in the
Albany~\cite{Albany} simulation package using numerous solver and discretization packages from the Trilinos
framework~\cite{TrilinosTOMS}.

\begin{figure}
\begin{center}
\includegraphics[scale=0.35]{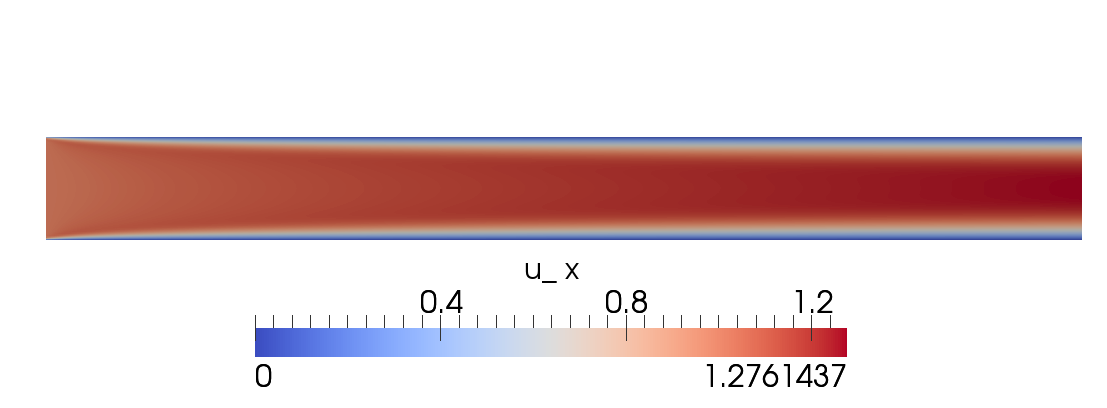}
\end{center}
\caption{Horizontal flow velocity at mean density and viscosity}\label{fig:flow_solution}
\end{figure}

We consider a problem where the ambient temperature and pressure are uncertain resulting in
uncertain density and viscosity. In particular we model the density and viscosity as uniformly distributed random
variables
\begin{equation}
  \rho \in [0.99\rho_0,\,1.01\rho_0] \qquad
  \mu \in [0.9\mu_0,\,1.1\mu_0].
\end{equation}
In other words, we assume density varies uniformly by 1\% and viscosity varies uniformly by 10\%. In the notation of
Section \ref{sec:setup}, we have 
\begin{align*}
\vx &= (\rho,\mu),\\
f(\vx) &= \frac{1}{Re} \;=\; \frac{\mu}{\rho u_0 W}. 
\end{align*}
The function $h(\vx)=g(f(\vx))$ corresponds to the maximum outflow velocity at the right side of the domain given fixed
values for $\rho$ and $\mu$. Each evaluation of $g$ involves an expensive solution of
equations~\ref{eq:ns_mom_nondim}-\ref{eq:ns_cont_nondim} -- compared to computing $f(\vx)$.

For this experiment, we choose a tensor product Gauss-Legendre quadrature rule with 11 points in the range of $\rho$ and
11 points in the range of $\mu$ for a total of $m=121$ points. We use the procedure from Section \ref{sec:setup} to
approximate the maximum outflow velocity at all 121 pairs of $(\rho,\mu)$ by constructing a $k$-point Gaussian
quadrature rule for $1/Re$ with $k=13$. In other words, with only 13 evaluations of $g$ -- the expensive flow solver --
we can approximate the output $h$ at 121 points in the parameter space corresponding to $\vx$. 

To check the error in the approximation, we also compute the maximum outflow velocity at all 121 combinations of $\rho$
and $\mu$, which enables the computation of \eqref{eq:lerror}. With 13 steps of the Lanczos procedure, we have a loss of
orthogonality in the basis vectors of $\tau=-13.14$ (see equation \eqref{eq:stop}). The error in approximation
(equation \eqref{eq:lerror}) is $1.55\times 10^{-6}$.

\section{Conclusion}

We have presented a method for approximating a composite function by implicitly approximating the outer function as a
polynomial of the output of the inner function. This measure transformation is based on Stieltjes' method for generating
orthogonal polynomials given an inner product, and it is implemented as Lanczos' method on a diagonal matrix of
inner function evaluations at the points of a discrete measure. We
have developed a heuristic for when to terminate the Lanczos iteration based on the loss of orthogonality in the Lanczos
vectors -- a common phenomenon for the algorithm in finite precision. The resulting method reduces the number of
evaluations of the outer function, which are only required at the Gaussian quadrature points of the transformed measure.
The numerical experiments show the behavior of the method and the scale of the reduction.





\bibliographystyle{elsarticle-num}
\bibliography{/home/paulcon/Dropbox/workspace/paulconstantine}







\end{document}